\documentclass[12pt]{article}

\usepackage[english]{babel}
\usepackage{graphicx}
\usepackage{latexsym}
\usepackage{amsfonts}
\usepackage{dsfont}
\usepackage{amscd}
\usepackage{amssymb}
\usepackage{hyperref}
\usepackage{amsmath}
\usepackage{mathtools}
\usepackage{mathrsfs}
\usepackage{amsthm}
\usepackage{tikz}

\input xy
\xyoption{all}

\newcommand{\ZZ}{\mathbb{Z}}

\newcommand{\s}{\mathcal{S}}
\newcommand{\Ef}{\widehat E_0}
\newcommand{\Et}{\widehat E_{\text{tight}}}
\newcommand{\Ct}{C^*_{\text{tight}}}
\newcommand{\Eu}{\widehat E_{\infty}}
\newcommand{\Gt}{\mathcal{G}_{\text{tight}}}
\newcommand{\G}{\mathcal{G}}
\newcommand{\EM}{\varnothing}
\newcommand{\J}{\mathcal{J}}
\newcommand{\I}{\mathcal{I}}
\newcommand{\Q}{\mathcal{Q}}

\theoremstyle{definition}
\newtheorem{theo}{Theorem}[section]

\newtheorem{ex}[theo]{Example}
\newtheorem{lem}[theo]{Lemma}
\newtheorem{defn}[theo]{Definition}

\newtheorem{prop}[theo]{Proposition}

\begin{document}
\date{}
\title{Boundary quotients of C*-algebras of right LCM semigroups}
\author{Charles Starling\thanks{Supported by CNPq (Brazil).}}

\maketitle

\begin{abstract}We study C*-algebras associated to right LCM semigroups, that is, semigroups which are left cancellative and for which any two principal right ideals are either disjoint or intersect in another principal right ideal. If $P$ is such a semigroup, its C*-algebra admits a natural boundary quotient $\mathcal{Q}(P)$. We show that $\mathcal{Q}(P)$ is isomorphic to the tight C*-algebra of a certain inverse semigroup associated to $P$, and thus is isomorphic to the C*-algebra of an \'etale groupoid. We use this to give conditions on $P$ which guarantee that $\mathcal{Q}(P)$ is simple and purely infinite, and give applications to self-similar groups and Zappa-Sz\'ep products of semigroups.
\end{abstract}

\tableofcontents
\section{Introduction}
A semigroup $P$ is left cancellative if $pq = ps$ implies that $q=s$, and C*-algebras associated to such semigroups are an active topic of research in operator algebras. Li's construction \cite{Li12} of a C*-algebra $C^*(P)$ from a left cancellative semigroup $P$ generalizes Nica's quasi-lattice ordered semigroups \cite{Ni92} and encompass a great deal of interesting C*-algebras, including the Cuntz algebras and the C*-algebra of the $ax+b$ semigroup, see \cite{Cu08}. Many semigroups of interest can be embedded into groups, and \cite{Li13} represents a comprehensive study of the C*-algebras of such semigroups. Another interesting class of semigroups (which has some overlap with the previous) are the {\em right LCM semigroups}, which are semigroups in which two principal right ideals are either disjoint or intersect in another principal right ideal. The paper \cite{BL14} considers the C*-algebras of such semigroups, and obtains many results about how the properties of $P$ influence the properties of $C^*(P)$ in this case.

We are concerned with boundary quotients of such algebras. Specifically, in \cite{BRRW14} the authors define a boundary quotient $\Q(P)$ of $C^*(P)$ \cite[Definition 5.1]{BRRW14} when $P$ is a right LCM semigroup, and this is the principal object of study in this paper. Quotients of this type are worthy of singling out because they frequently give examples of simple C*-algebra where the original would not (unless of course they are equal), and examples of simple C*-algebras are of interest to C*-algebra classification.


It turns out that this boundary quotient can be studied by using work of Exel \cite{Ex08} on inverse semigroup C*-algebras. An {\em inverse semigroup} is a semigroup $S$ such that for each $s\in S$ there is a unique $s^*\in S$ such that $ss^*s = s$ and $s^*ss^* = s^*$, and one can define a universal C*-algebra for representations of $S$, as in \cite{Pa99}. Further work of Norling \cite{No14} determined that for a left cancellative semigroup $P$, $C^*(P)$ is isomorphic to the universal C*-algebra of a certain inverse semigroup (denoted $\s$ in the sequel) obtained from $P$. In the paper \cite{Ex08} Exel discovered a natural quotient for Paterson's inverse semigroup algebra called the {\em tight C*-algebra} of an inverse semigroup. This C*-algebra is universal for so called {\em tight} representations of of the inverse semigroup; representations of this kind enforce a kind of nondegeneracy condition. This construction has been studied by many authors, see \cite{EGS12}, \cite{SM11}, \cite{EPSep14}, \cite{Ste14}. Our first main result, Theorem \ref{maintheorem}, states that the tight C*-algebra of $\s$ is isomorphic to the boundary quotient $\Q(P)$. This generalizes a combination of \cite[Corollary 6.4]{EP13} and \cite[Theorem 6.7]{BRRW14} from the case of self-similar groups, and in fact it was the desire to generalize this result to other types of semigroups which was the motivation for this work.

Both Paterson's and Exel's C*-algebras can be presented as the C*-algebras of certain \'etale groupoids, and so can be analyzed by using the many results concerning \'etale groupoids in the literature. There is a small difficulty in doing so however, because the groupoids which arise in this way can be non-Hausdorff, and a majority of the results in the literature about the structure of \'etale groupoid C*-algebras assumes the Hausdorff property. One condition which guarantees Hausdorff is right cancellativity (in addition to the left cancellativity already assumed), but one can weaken this a bit to obtain a condition on $P$ which is equivalent to the groupoid being Hausdorff, see Proposition \ref{hausdorff}. Here, we employ the results in \cite{EP14}, \cite{BCFS14}, and \cite{AD97} to find conditions on $P$ which guarantee that $\Q(P)$ is simple and purely infinite. Of specific use is \cite{EP14}, because that paper is concerned with \'etale groupoids arising from inverse semigroup actions. We note that in the recent work \cite{Ste14}, Steinberg independently comes to many of the same conclusions as \cite{EP14}, and many of the results we use from \cite{EP14} also appear in \cite{Ste14}, but throughout this article we will reference their appearance in \cite{EP14}.

This paper is organized as follows. Section \ref{background} recalls definitions and sets notation. In Section \ref{boundary}, we establish an isomorphism between $\Q(P)$ and the tight C*-algebra of an inverse semigroup, and in Section \ref{properties} we deduce properties of $\Q(P)$ by using this realization of it as a C*-algebra of an \'etale groupoid. In Section \ref{examplessection} we give some examples of right LCM semigroups from the literature, including some arising from Zappa-Sz\'ep products of semigroups and from self-similar groups. Finally, in Appendix \ref{coreappendix} we prove a small result about inverse semigroup actions which generalizes Proposition \ref{EPcore} and which may be of independent interest.

{\bf Acknowledgements}: We would like to thank Ruy Exel for many helpful and enlightening conversations about this work.

\section{Background}\label{background}

Let $P$ be a semigroup. We say that $P$ is {\em left cancellative} if $pq = ps$ implies that $q = s$ for all $p,q,$ and $s\in P$. A {\em right ideal} of $P$ is a set $X\subset P$ such that $XP = \{xp\mid x\in X, p\in P\}$ is a subset of $X$. If $p\in P$, then the set $pP = \{pq\mid q\in P\}$ is a right ideal, and any right ideal of this form is called a {\em principal right ideal}. An element of $pP$ is called a {\em right multiple} of $p$. All semigroups are assumed to be countable and discrete.

If $X$ is a right ideal of $P$, then for all $p\in P$ the sets
\[
pX = \{px\mid x\in X\}, \hspace{1cm}p^{-1}X = \{ y\in P\mid py\in X\}
\] 
are also right ideals. We let $\mathcal{J}(P)$ denote the smallest set of right ideals which contains $P$ and $\EM$, is closed under intersections, and such that $X\in \mathcal{J}(P)$ and $p\in P$ implies that both $pX$ and $p^{-1}X$ are in $\mathcal{J}(P)$. Then $\mathcal{J}(P)$ is a semilattice under intersection, and is called the semilattice of {\em constructible ideals}. For a left cancellative semigroup $P$, Li constructs a C*-algebra $C^*(P)$. 
\begin{defn}\label{LiDef}
Let $P$ be a left cancellative semigroup, and let $\J(P)$ denote the set of constructible ideals of $P$. Then $C^*(P)$ is defined to be the universal C*-algebra generated by a set of isometries $\{ v_p\mid p\in P\}$ and a set of projections $\{e_X\mid X\in \J(P)\}$
subject to the following:
\begin{enumerate}\addtolength{\itemsep}{-0.5\baselineskip}
\item[(L1)]$v_pv_q = v_{pq}$ for all $p, q\in P$, 
\item[(L2)]$v_pe_Xv^*_{p} = e_{pX}$ for all $p\in P$ and $X\in \J$,
\item[(L3)]$e_P = 1$ and $e_\EM = 0$, and
\item[(L4)]$e_Xe_Y = e_{X\cap Y}$ for all $X, Y\in \J$.
\end{enumerate}
\end{defn}

It is clear that $\mathcal{J}(P)$ contains every principal right ideal. In this paper we consider the following class of semigroups for which $\mathcal{J}(P)$ is equal to the set of principal right ideals. 
\begin{defn}
A semigroup $P$ is called a {\em right LCM semigroup} if it is left cancellative and the intersection of any two principal right ideals is either empty or another principal right ideal.
\end{defn}
Semigroups of this type have gone by other names in the literature. In \cite{La99}, Lawson considers the dual definition (ie, right cancellative semigroups such that two principal right ideals are either disjoint or intersect in another principal right ideal) and calls these {\em CRM monoids} (named for Clifford, Reilly and McAlister). In other works such as \cite{No14}, such semigroups are said to satisfy {\em Clifford's condition}.

Let $P$ be a right LCM semigroup with identity, and let $U(P)$ denote the invertible elements of $P$ (invertible elements of $P$ are also sometimes called the {\em units} of $P$). Then if we have $p, q\in P$ such that $pP\cap qP = rP$, we see that every element of $P$ which is right multiple of both $p$ and $q$ is also a right multiple of $r$, and we say that $r$ is a {\em right least common multiple} (or {\em right LCM}) of $p$ and $q$. If $rP = sP$, then a short calculation shows that there must exist $u\in U(P)$ such that $ru = s$. Hence, if $r$ is a right LCM of $p$ and $q$ then so is $ru$ for all $u\in U(P)$. Also, if $pP \cap qP = rP$, then there exist $p', q'\in P$ such that $pp' = qq' = r$. This right least common multiple property is the source of the terminology ``right LCM''.

Let $P$ be a right LCM semigroup and suppose that we have $p, q\in P$ such that $pP\cap qP = rP$ with $pp' = qq' = r$. Then it is straightforward to show that $p^{-1}qP = p'P$, and so the set of principal right ideals is in fact equal to the set of constructible ideals.

We shall be concerned with groupoids constructed from semigroups. Recall that a {\em groupoid} consists of a set $\G$ together with a subset $\G^{(2)} \subset \G \times \G$, called the set of composable pairs, a product map $\G^{(2)} \to \G$ with $(\gamma, \eta)\mapsto \gamma\eta$, and an inverse map from $\G$ to $\G$ with $\gamma \mapsto \gamma^{-1}$ such that
\begin{enumerate}\addtolength{\itemsep}{-0.5\baselineskip}
\item $(\gamma^{-1})^{-1} = \gamma$ for all $\gamma\in \G$,
\item If $(\gamma, \eta), (\eta, \nu)\in \G^{(2)}$, then $(\gamma\eta,\nu), (\gamma, \eta\nu)\in \G^{(2)}$ and $(\gamma\eta)\nu = \gamma(\eta\nu)$,
\item $(\gamma, \gamma^{-1}), (\gamma^{-1},\gamma)\in \G^{(2)}$, and $\gamma^{-1}\gamma\eta = \eta$, $\xi\gamma\gamma^{-1}$ for all $\eta, \xi$ with $(\gamma, \eta), (\eta,\xi) \in \G^{(2)}$.
\end{enumerate}
The set of {\em units} of $\G$ is the subset $\G^{(0)}$ of elements of the form $\gamma\gamma^{-1}$. The maps $r: \G\to \G^{(0)}$ and $d:\G\to \G^{(0)}$ given by
\[
r(\gamma) = \gamma\gamma^{-1}, \hspace{1cm} d(\gamma) = \gamma^{-1}\gamma
\]
are called the {\em range} and {\em source} maps respectively. It is straightforward to check that $(\gamma, \eta)\in \G^{(2)}$ is equivalent to
$r(\eta) = d(\gamma)$. 

One thinks of a groupoid $\G$ as a set of ``arrows'' between elements of $\G^{(0)}$. Given $x\in \G^{(0)}$, let
\[
\G^x := r^{-1}(x), \hspace{1cm} \G_x := d^{-1}(x), \hspace{1cm} \G^{x}_x := d^{-1}(x)\cap r^{-1}(x),
\] 
which are thought of, respectively, as the arrows ending at $x$, the arrows beginning at $x$, and all the arrows both ending and beginning at $x$. For all $x\in \G^{(0)}$, $\G^{x}_x$ is a group with identity $x$ when given the operations inherited from $\G$, and is called the {\em isotropy group} of $x$. The set Iso$(\G) = \cup_{x\in \G^{(0)}}\G_x^x$ is called the {\em isotropy group bundle} of $\G$. The {\em orbit} of $x\in \G^{(0)}$ is the set $\G(u):= r(\G_x) = s(\G^x)$.

A {\em topological groupoid} is a groupoid which is a topological space where the inverse and
product maps are continuous, where we are considering $\G^{(2)}$ with the product topology inherited from $\G\times\G$. Two topological groupoids are said to be {\em isomorphic} if there is a homeomorphism between them which preserves the inverse and product operations. A groupoid with topology $\G$ is called {\em \'etale} if it is locally compact, second countable, and the maps $r$ and $d$ are local homeomorphisms. These properties imply that $\G^{(0)}$ is open in $\G$ and that for all $x\in \G^{(0)}$ the spaces $\G^x$ and $\G_x$ are discrete.

For subsets $S, T\subset \G$, let $ST = \{\gamma\eta\mid \gamma\in S, \eta\in T, d(\gamma) = r(\eta)\}$. A subset $S\subset \G$ of a topological groupoid is called a {\em bisection} if the restrictions of $r$ and $d$ to $S$ are both injective. In an \'etale groupoid $\G$, the collection of open bisections forms a basis for the topology of $\G$, cf. \cite[Proposition 3.5]{Ex08}. If $S$ and $T$ are bisections in an \'etale groupoid, then so is $ST$.

A subset $U\subset \G^{(0)}$ is called {\em invariant} if for all $\gamma\in \G$, $r(\gamma)\in U$ implies that $s(\gamma)\in U$. A topological groupoid is called {\em minimal} if the only nonempty open invariant subset of $\G^{(0)}$ is $\G^{(0)}$. We say that $\G$ is {\em topologically principal} if the set of $x\in \G^{(0)}$ for which $\G^x_x = \{x\}$ is dense. We will say that $\G$ is {\em essentially principal} if the interior of Iso$(\G)$ is equal to $\G^{(0)}$, and we will say that $\G$ is {\em effective} if the interior of Iso$(\G)\setminus \G^{(0)}$ is empty. When $\G$ is a locally compact, second countable, Hausdorff, \'etale groupoid, then 
\[
\G \text{ topologically principal } \Leftrightarrow \G \text{ essentially principal } \Leftrightarrow \G \text{ effective},
\]
see \cite[Proposition 3.1]{R80} and \cite[Lemma 3.1]{BCFS14}.

In a construction from \cite{R80}, to an \'etale groupoid $\G$ one can associate C*-algebras $C^*(\G)$ and $C^*_r(\G)$, called the {\em C*-algebra of $\G$} and the {\em reduced C*-algebra of $\G$} respectively. To build these C*-algebras one starts with $C_c(\G)$, the continuous compactly supported functions on $\G$, which becomes a complex $*$-algebra when given the convolution product and involution given by
\[
f\star g(\gamma) = \sum_{\gamma_1\gamma_2 = \gamma}f(\gamma_1)g(\gamma_2)\hspace{1cm} f^*(\gamma) = \overline{f(\gamma^{-1})}.
\] 
We note that the fact that $\G$ is \'etale implies that the sum above is finite. In this work we are not concerned with the specifics, but both $C^*(\G)$ and $C^*_r(\G)$ are completions of $C_c(\G)$ in suitable norms, and $C^*_r(\G)$ is always a quotient of $C^*(\G)$. For more details, the interested reader is directed to \cite{R80}. In this paper we construct \'etale groupoids from certain semigroups, and most of our results concern the case when $C^*(\G) = C^*_r(\G)$. This happens for instance when $\G$ is {\em amenable}, see \cite{AR00}.

We will use the following result to deduce properties of the C*-algebras we construct.
\begin{theo}\cite[Theorem 5.1]{BCFS14}\label{groupoidsimple}
Let $\G$ be a second countable locally compact Hausdorff \'etale groupoid. Then $C^*(\G)$ is simple if and only if the following conditions are satisfied:
\begin{enumerate}\addtolength{\itemsep}{-0.5\baselineskip}
\item $C^*(\G) = C^*_r(\G)$,
\item $\G$ is topologically principal, and
\item $\G$ is minimal. 
\end{enumerate}
\end{theo}

An \'etale groupoid is called {\em locally contracting} if for every nonempty open subset $U\subset \G^{(0)}$, there exists an open subset $V\subset U$ and an open bisection $S\subset \G$ such that $\overline{V}\subset S^{-1}S$ and $SVS^{-1}\subsetneq V$. By \cite[Proposition 2.4]{AD97}, if $C^*_r(\G)$ is simple and $\G$ is locally contracting, then $C^*_r(\G)$ is purely infinite. We assume knowledge of C*-algebras, but for the unfamiliar an excellent reference for the undefined terms above is \cite{Dav}.

\section{The boundary $\Q(P)$ as the tight C*-algebra of an inverse semigroup}\label{boundary}

Recall that a semigroup $S$ is called {\em regular} if for all $s\in S$ there exists an element $t\in S$ such that $tst = t$ and $sts = s$. Such an element $t$ is often called an {\em inverse} of $s$, though even if $S$ has an identity we need not have $ts = 1$. However, we always have $(ts)^2 =tsts = ts$, that is to say that $ts$ is idempotent. We let $E(S) = \{ e\in S\mid e^2=e\}$ denote the set of all idempotent elements of $S$. A regular semigroup is called an {\em inverse semigroup} if each element $s$ has a unique inverse, denoted $s^*$. It is a fact that a regular semigroup is an inverse semigroup if and only if elements of $E(S)$ commute, and we note that in this case $E(S)$ is closed under multiplication.  

\begin{ex}\label{PCM}
We give an important and fundamental example of an inverse semigroup. Let $X$ be a set. Consider
\[
\I(X) = \{f:U\to V\mid U, V\subset X, f\text{ is bijective}\}.
\]
Then $\I(X)$ is an inverse semigroup when given the operation of function composition on the largest domain possible. The inverse of an element $f: U\to V$ is the inverse function $f^* = f^{-1}:V\to U$. One sees that the identity function is the identity for this inverse semigroup, and more generally every idempotent is the identity on some subset. If we have $f, g\in \I(X)$ such that the range of $f$ does not intersect the domain of $g$, then the composition $g\circ f$ on the largest domain possible is equal to the empty function, which acts as a zero element in $\I(X)$. It is an important fact in semigroup theory that every inverse semigroup can be embedded into $\I(X)$ for some set $X$ -- this is known as the Wagner-Preston theorem. 

This example demonstrates that many inverse semigroups naturally contain a zero element. Because of this, the two algebraic objects we consider in this paper, namely ``right LCM semigroups'' and ``inverse semigroups'' should be thought of as quite different types of objects, as left cancellativity in a right LCM semigroup eliminates the possibility of a zero element in nontrivial cases.
\end{ex}

Now, given a right LCM semigroup $P$ we will construct an inverse semigroup $\s$. We define an equivalence relation $\sim$ on $P\times P$ by saying that $(p, q)\sim (a, b)$ if and only if there exists $u\in U(P)$ such that $au = p$ and $bu = q$. In other words, the equivalence class of $(p, q)$ consists of all elements of the form $(pu, qu)$ with $u\in U(P)$. Denote by $[p, q]$ the equivalence class of $(p, q)$. The following proposition is essentially \cite[Theorem 3]{La99}, though Lawson presents the dual construction for right cancellative left LCM semigroups.
\begin{prop}
Let $P$ be a right LCM semigroup with identity $1_P$, and let
\begin{equation}\label{Sdef}
\s := \{ [p,	q]\mid p, q\in P\}\cup \{0\}.
\end{equation}  Then $\s$ becomes an inverse semigroup with identity $1_\s = [1_P, 1_P]$ when given the operation
\[
[a,b][c,d] = \begin{cases}[ab', dc'] & \text{if } cP\cap bP = rP \text{ and } cc' = bb' = r \\ 0 & \text{if }cP\cap bP = \EM\end{cases}
\]
and $s0 = 0s = 0$ for all $s\in \s$. In this case, we have that $[a,b]^* = [b,a]$ and 
\[
E(\s) = \{[a, a]\mid a\in P\}\cup \{0\}.
\]
\end{prop}
\begin{proof}
Before we start, we note that although this proof is straightforward it is long and tedious. However, it may be valuable if one wishes to get a feel for right LCM semigroups.

We first show that the multiplication above is well-defined. Suppose that $[a,b],[c,d]\in \s$. Then if $u, v\in U(P)$, we know that $buP = bP$ and $cvP = cP$, and so $[a,b][c,d] \neq 0$ if and only if $[au, bu][cv, dv]\neq 0$. So, suppose that $[a,b][c,d] = [ab',dc']$, where $bP\cap cP = rP$ with $bb' = cc' = r$. Then $buP\cap cvP = bP\cap cP = rP$, and so there exist $b'', c''\in P$ such that $bub'' = cvc'' = r$. Because $bb' = cc' = r$ and $P$ is left cancellative, we have that $ub'' = b'$ and $vc'' = c'$. Hence
\[
[au,bu][cv,dv] = [aub'', dvc''] = [ab', dc'] = [a,b][c,d]
\]
and so the multiplication is well-defined.

We now show that the given multiplication is associative. Take $[a,b], [c,d], [e,f]\in \s$ and first suppose that $[a,b]\Big([c,d][e,f]\Big)\neq 0$. Then there must be $r_1, d_1, e_1\in P$ such that $dP\cap eP = r_1P$, $dd_1 = ee_1 = r_1$, and $[c,d][e,f] = [cd_1, fe_1]$. Since we assumed that $[a,b][cd_1, fe_1] \neq 0$, we now must have that there exist $r_2, b_1, c_1\in P$ such that $bP\cap cd_1P = r_2P$, $bb_1 = cd_1c_1 = r_2$, and 
\[
[a,b]\Big([c,d][e,f]\Big) = [ab_1, fe_1c_1] \neq 0.
\]
Since $bP\cap cd_1P\neq \EM$, we must have that $bP\cap cP\neq \EM$, so there exists $r_3, b_2, c_2\in P$ such that $bP \cap cP = r_3P$, $bb_2 = cc_2 = r_3$, and $[a, b][c,d] = [ab_2, dc_2]$. In addition, we have that $r_2P\subset r_3P$, and so there exists $q\in P$ such that $r_2 = r_3q$. Now we have that
\begin{equation}\label{d1c1c2q}
cd_1c_1 = r_2 = r_3q = cc_2q \hspace{0.5cm} \Rightarrow \hspace{0.5cm} d_1c_1 = c_2q
\end{equation}
and so we have
\[
ee_1c_1 = dd_1c_1 = dc_2q \hspace{0.5cm} \Rightarrow \hspace{0.5cm} dc_2P\cap eP \neq \EM.
\]
Hence $\Big([a,b][c,d]\Big)[e,f]=[ab_2, dc_2][e,f]\neq 0$. Furthermore, there exist $r_4, d_2, e_2\in P$ such that $dc_2P\cap eP = r_4P$, $dc_2d_2 = ee_2 = r_4$, and 
\[
\Big([a,b][c,d]\Big)[e,f] = [ab_2, dc_2][e,f] = [ab_2d_2, fe_2].
\]
In this case, we have that $r_4P\subset r_1P$, and so there exists $p\in P$ such that $r_4 = r_1p$. Also, similar to \eqref{d1c1c2q}, we have that $c_2d_2 = d_1p$. If instead we started by insisting that $\Big([a,b][c,d]\Big)[e,f]\neq 0$, then a similar argument gives that $[a,b]\Big([c,d][e,f]\Big) \neq 0$. Thus to show associativity we can assume both products are nonzero and that we have elements $b_1, b_2, c_1, c_2, d_1, d_2, e_1, e_2, r_1, r_2, r_3, r_4, q, p\in P$ such that
\[
\begin{array}{l}
dP\cap eP = r_1P\\
bP \cap cd_1P = r_2P\\
bP \cap cP = r_3P\\
dc_2P\cap eP = r_4P
\end{array}
\hspace{1cm}
\begin{array}{l}
dd_1 = ee_1 = r_1\\
bb_1 = cd_1c_1 = r_2\\
bb_2 = cc_2 = r_3\\
dc_2d_2 = ee_2 = r_4
\end{array}
\hspace{1cm}
\begin{array}{l}
r_2 = r_3q\\
r_4 = r_1p\\
d_1c_1 = c_2q\\
c_2d_2 = d_1p
\end{array}
\]
Now, we have that $r_1c_1 = ee_1c_1 = dd_1c_1 = dc_2q$, and so $r_1c_1\in dc_2P\cap eP = r_4P$, meaning that there exists $k_1\in P$ such that $r_1c_1 = r_4k_1 = r_1pk_1$, and because $P$ is left cancellative we have that $pk_1 = c_1$.

Similarly, $r_3d_2 = bb_2d_2 = cc_2d_2 = cc_1d_1$, and so $r_3d_2\in bP\cap cd_1P = r_2P$. This means that there exists $k_2\in P$ such that $r_3d_2 = r_2k_2 = r_3qk_2$, and since $P$ is left cancellative we have that $d_2 = qk_2$.

We claim that $k_1k_2 = 1 = k_2k_1 = 1_P$, and hence $k_1, k_2\in U(P)$. We calculate
\[
d_1pk_1k_2 = d_1c_1k_2 = c_2qk_2 = c_2d_2 = d_1p \hspace{0.5cm} \Rightarrow \hspace{0.5cm} k_1k_2 = 1_P,
\]
\[
c_2qk_2k_1 = c_2d_2k_1 = d_1pk_1 = d_1c_1 = c_2q \hspace{0.5cm} \Rightarrow \hspace{0.5cm} k_2k_1 = 1_P.
\]
Now, $bb_1 = r_2 = r_3q = bb_2q$, and so $b_1 = b_2q$. Similarly, $e_2 = e_1p$. Thus,
\[
[a,b]\Big([c,d][e,f]\Big) = [ab_1, fe_1c_1] = [ab_2q, fe_1c_1],
\]
\[
\Big([a,b][c,d]\Big)[e,f] = [ab_2d_2, fe_2] =[ab_2d_2, fe_1p],
\]
and 
\[
ab_2d_2k_1 = ab_2qk_2k_1 = ab_2q, \hspace{0.5cm}fe_1pk_1 = fe_1c_1.
\]
Hence $[a,b]\Big([c,d][e,f]\Big) = \Big([a,b][c,d]\Big)[e,f]$ as required.

Suppose now that $[p,q][p,q] = [p,q]$. Then $pP\cap qP = rP$ and there exist $p', q'$ such that $pp'= qq' = r$, and $[p,q]= [p,q][p,q] = [pp',qq'] = [r,r]$. Hence there exists $u\in U(P)$ such that $p= ru = q$. Hence the only idempotent elements of $\s$ are elements of the form $[p,p]$, together with the 0 element. Now suppose that we have $p, q\in P$ such that $[p, p][q,q] \neq 0$. Then $pP\cap qP = rP$ for some $r\in P$ and there exist $p', q'$ such that $pp'=qq' = r$, and $[p,p][q,q] = [pp',qq'] = [r,r]$. It is clear that this is equal to $[q,q][p,p]$, and that $[p, p][q,q]=0$ if and only if $[q,q][p,p] = 0$. Hence the idempotents of $\s$ commute.

It is obvious that $[p,q][q,p][p,q] = [p,q]$ and $[q,p][p,q][q,p] = [q,p]$. Hence each element of $\s$ has an inverse (0 is the inverse of 0), and so $\s$ is regular. As above, the idempotents of $\s$ commute, hence $\s$ is an inverse semigroup.
\end{proof}

There is another formulation of the semigroup $\s$ above, considered for example in \cite{No14}. Consider $\I(P)$ as in Example \ref{PCM}. Since $P$ is assumed to be left cancellative, the map $\lambda_p:P\to pP$ defined by $\lambda_p(q) = pq$ is a bijection, and hence an element of $\I(P)$. Let $\I_l(P)$ denote the inverse semigroup generated by the elements $\{\lambda_p\}_{p\in P}$ inside $\I(P)$. This is sometimes called the {\em left inverse hull} of $P$. Then the map from $\s$ to $\I_l(P)$ given by $[p,q]\mapsto \lambda_p\lambda_q^{-1}$ is an isomorphism.

The main result of this section is an isomorphism between two C*-algebras, denoted in the sequel $\Q(P)$ and $C^*_{\text{tight}}(\s)$. We begin by defining $\Q(P)$. A finite set $F\subset P$ is called a {\em foundation set} if for all $p\in P$ there exists $f\in F$ such that $fP\cap pP \neq \EM$. The following definition is \cite[Definition 5.1]{BRRW14}.

\begin{defn}Let $P$ be a right LCM semigroup. The {\em boundary quotient} of $C^*(P)$, denoted $\Q(P)$ is the universal C*-algebra generated by sets  $\{ v_p\mid p\in P\}$ and  $\{e_X\mid X\in \J\}$ subject to the relations (L1)--(L4) in Definition \ref{LiDef} and 
\[
\prod_{p\in F} (1-e_{pP}) = 0\text{ for all foundation sets }F\subset P.
\]
\end{defn}

Now that we have defined $\Q(P)$, we define the second algebra which concerns us. Let $A$ be a C*-algebra and let $S$ be an inverse semigroup with zero. Then a {\em representation} of $S$ is a map $\pi: S\to A$ such that for all $s, t\in S$ we have $\pi(st) = \pi(s)\pi(t)$, $\pi(s^*) = \pi(s)^*$, and $\pi(0) = 0$. The {\em universal C*-algebra of $S$}, considered in \cite{Pa99} and denoted $C^*_u(S)$, is the universal C*-algebra generated by one element for each element of $S$ such that the standard map $\pi_u: S\to C^*_u(S)$ is a representation. Note that this implies that $\pi_u(s)$ is a partial isometry for each $s\in S$. 

 Let $S$ be an inverse semigroup, let $\pi: S\to A$ be a representation, and let $D_\pi$ denote the C*-subalgebra of $A$ generated by $\pi(E(S))$. Since $E(S)$ is commutative, $D_\pi$ must be a commutative C*-algebra. The set 
\[
\mathscr{B}_\pi = \{e\in D_\pi\mid e^2 = e\}
\]
is a Boolean algebra when given the operations
\[
e \wedge f = ef \hspace{0.5cm} e\vee f = e+f - ef\hspace{0.5cm} \neg e = 1-e
\]

We will recall a subclass of representations defined in \cite{Ex08}. Let $S$ be an inverse semigroup and let $F\subset Z\subset E(S)$. We say that $F$ is a {\em cover} of $Z$ if for every nonzero $z\in Z$ there is $f\in Z$ such that $fz \neq 0$. If $x\in E(S)$ and $F$ is a cover for $\{y\in E(S)\mid yx = x\}$, then we say that $F$ is a cover of $x$. For finite sets $X, Y\subset E(S)$, let
\[
E(S)^{X, Y} = \{e\in E(S)\mid ex = e \text{ for all } x\in X\text{ and }ey = 0\text{ for all } y\in Y\}
\]
A representation $\pi: S\to A$ is called {\em tight} if for every pair of finite sets $X, Y\subset E(S)$ and every finite cover $Z$ of $E(S)^{X,Y}$, we have
\[
\bigvee_{z\in Z}\pi(z) = \prod_{x\in X}\pi(x)\prod_{y\in Y}(1-\pi(y)).
\]
The {\em tight C*-algebra of $S$}, denoted $C^*_{\text{tight}}(S)$,  is the universal C*-algebra generated by one element for each element of $S$ subject to the relations saying that the standard map $\pi_t: S\to C^*_{\text{tight}}(S)$ is a tight representation.

\begin{lem}\label{starnotcalc}Let $P$ be a right LCM semigroup, and let $\{ v_p\mid p\in P\}$ and  $\{e_X\mid X\in \J(P)\}$ be as in Definition \ref{LiDef}. Then for all $p,q\in P$ we have
\[
v_p^*v_q = \begin{cases}v_{p'}v_{q'}^*&\text{if }pP\cap qP = rP\text{ and }r = pp' = qq'\\0&\text{if }pP\cap qP = \EM.\end{cases}.
\]
\end{lem}
\begin{proof}
Suppose that $pP\cap qP = rP\text{ and }r = pp' = qq'$. Then 
\begin{eqnarray*}
v_p^*v_q & =& (v_p^*e_{pP})(e_{qP}v_q) = v_p^*e_{rP}v_q\\
&=& v_p^*v_rv_r^*v_q = v_p^*(v_pv_{p'})(v_qv_{q'})^*v_q = (v_p^*v_p)v_{p'}v_{q'}^*(v_{q'}^*v_q) =  v_{p'}v_{q'}^*. 
\end{eqnarray*}
The second equality above shows that $v_p^*v_q = 0$ if $pP\cap qP = \EM$. 
\end{proof}

\begin{lem}\label{L1}
Let $P$ be a right LCM semigroup with identity, and let $\s$ be as in \eqref{Sdef}. Then the map $\pi:\s \to \Q(P)$ defined by
\[
\pi([p,q]) = v_pv_q^*
\]
\[
\pi(0) = 0
\]
is a tight representation of $\s$.
\end{lem}
\begin{proof}
It is easy to see from Lemma \ref{starnotcalc} and \eqref{Sdef} that the map $\pi$ above is a representation of $\s$. Now, suppose that $F$ is a foundation set. Then, by de Morgan's laws in a Boolean algebra we have
\[
0 = \prod_{f\in F}(1-e_{fP}) = \bigwedge_{f\in F}(\neg \pi([f,f])) = \neg\left(\bigvee_{f\in F} \pi([f,f])\right) = 1-\bigvee_{f\in F} \pi([f,f])
\]
and so $\bigvee_{f\in F} \pi([f,f]) = 1$. Hence by \cite[Proposition 11.8]{Ex08}, to show that $\pi$ is tight we need only check that for every $[p,p]\in E(\s)$ and every finite cover $Z$ of $[p,p]$, we have the equality
\[
\bigvee_{z\in Z}\pi(z) = e_{pP}. 
\]
So, take $p\in P$ and suppose that $Z$ is a finite cover of $[p,p]$. For $Z$ to be a finite cover of $[p,p]$, we must have that for all  $z\in Z$, $zP \subset pP$ and whenever we have $q\in P$ such that $qP\subset pP$, there exists $z\in Z$ such that $qP\cap zP \neq \EM$. The first condition implies that for all $z\in Z$, there exists $a_z\in P$ such that $z = pa_z$. We claim that $\{a_z\mid z\in Z\}$ is a foundation set. Indeed, for every $q\in P$, we have that $pqP\subset pP$, and so there exists $z\in Z$ such that $zP\cap pqP = pa_zP \cap pqP\neq  \EM$, and so $a_zP\cap qP \neq \EM$. Hence $1 = \bigvee_{z\in Z}e_{a_zP}$. Hence we have
\begin{eqnarray*}
e_{pP} &=& v_p1v_p^*\\
       &=& v_p\left(\bigvee_{z\in Z}e_{a_zP}\right)v_p^*\\
       &=& \bigvee_{z\in Z}v_pe_{a_zP}v_p^*\\
       &=& \bigvee_{z\in Z}e_{pa_zP}\\
       &=& \bigvee_{z\in Z}\pi(z)
\end{eqnarray*}
\end{proof}

\begin{lem}\label{L2}
Let $P$ be a right LCM semigroup with identity, let $\s$ be as in \eqref{Sdef}, and let $\pi$ be any tight representation of $\s$. Then for every foundation set $F\subset P$, 
\[
\prod_{f\in F}(1- \pi([f,f])) = 0.
\]
\end{lem}
\begin{proof}
Let $F\subset P$ be a foundation set. Again, by de Morgan's laws in a Boolean algebra, we have
\[
\prod_{f\in F}(1- \pi([f,f])) = \bigwedge_{f\in F}(\neg \pi([f,f])) = \neg(\bigvee_{f\in F} \pi([f,f])) = 1 - \bigvee_{f\in F} \pi([f,f]).
\]
Hence we will be done if we can show that $\bigvee_{f\in F} \pi_t([f,f]) = 1$. Let $X = \{1_\s\}$ and $Y= \EM$. Then 
\[
E(\s)^{X, Y} = \{ e\in E(\s) \mid e1_\s = e\} = E(\s)
\] 
and since $F$ is a foundation set, $Z = \{ [f, f]\mid f\in F\}$ is a finite cover for $E(\s)$. Thus we have
\begin{eqnarray*}
\bigvee_{z\in Z}\pi(z) &=& \prod_{x\in X}\pi(x)\prod_{y\in Y}(1-\pi(y))\\
\Rightarrow \hspace{0.5cm} \bigvee_{f\in F}\pi([f,f]) & = & \pi(1)\\
&=& 1.
\end{eqnarray*}
\end{proof}
The above two lemmas combine to give the main result of this section.
\begin{theo}\label{maintheorem}
Let $P$ be a right LCM semigroup with identity, and let $\s$ be as in \eqref{Sdef}. Then there is an isomorphism $\Phi: \Q(P)\to \Ct(\s)$ such that $\Phi(v_pv_q^*) = \pi_t([p,q])$ for all $p,q\in P$.
\end{theo}
\begin{proof}
By Lemma \ref{L1} and the fact that $\Ct(\s)$ is universal for tight representations of $\s$, there exists a $*$-homomorphism $\Phi_\pi: \Ct(\s) \to \Q(P)$ such that $\Phi_\pi\circ \pi_t([p, q]) = v_pv_q^*$. Conversely, by Lemma \ref{L2} and the universal property of $\Q(P)$, there exists a $*$-homomorphism $\Phi: \Q(P)\to \Ct(\s)$ such that $\Phi(v_pv_q^*) = \pi_t([p,q])$. Hence $\Phi\circ\Phi_\pi$ is the identity on $\Q(P)$, $\Phi_\pi\circ\Phi$ is the identity on $\Ct(\s)$, and so $\Phi$ is an isomorphism.
\end{proof}

\section{Properties of $\Q(P)$}\label{properties}
One of the consequences of Theorem \ref{maintheorem} is that $\Q(P)$ is isomorphic to the C*-algebra of an \'etale groupoid, and we may therefore study $\Q(P)$ by studying the groupoid.

\subsection{$\Q(P)$ as a groupoid C*-algebra}\label{groupoidsubsection}
We now review the construction of the tight groupoid of an inverse semigroup. For more, the interested reader is directed to \cite{Ex08}. 

Let $S$ be an inverse semigroup. There is a natural partial order on $S$ given by $s \leqslant t$ if and only if $s = ts^*s$. If $e, f\in E(S)$, $e \leqslant f$ if and only if $ef = e$. This partial order is best understood in the context of the inverse semigroup $\mathcal{I}(X)$ -- here we have $\varphi \leqslant \psi$ if and only if $\psi$ extends $\varphi$ as a function.

In this order, each pair $e, f\in E(S)$ has a unique greatest lower bound, namely their product $ef$. Hence, with the order above $E(S)$ is a semilattice. If $S$ has an identity $1_S$, then it is the unique maximal element of $E(S)$, and if $S$ has a zero element it is the unique minimal element of $E(S)$. 

A {\em filter} in $E(S)$ is a {\em proper} subset $\xi\subset E(S)$ which is {\em downwards directed} in the sense that $e, f\in \xi$ implies that $ef\in \xi$, and {\em upwards closed} in the sense that if $e\in \xi$, $f\in E(S)$ and $e\leqslant f$ implies that $f\in \xi$. If a subset $\xi\subset E(S)$ is proper and downwards directed it is called a {\em filter base}, and the set 
\[
\overline{\xi} = \{e\in E(S)\mid f\leqslant e\text{ for some }f \in \xi\},
\]
called the {\em upwards closure} of $\xi$, is a filter. A filter is called an {\em ultrafilter} if it is not properly contained in another filter. Ultrafilters always exist by Zorn's Lemma. 

We let $\Ef(S)$ denote the set of filters in $E(S)$. This set has a natural topology given by seeing it as a subspace of $\{0,1\}^{E(S)}$ with the product topology. There is a convenient basis for this topology: for finite sets $X, Y\subset E(S)$, let
\[
U(X,Y) = \{ \xi \in \Ef(S)\mid x\in \xi\text{ for all }x\in X, y\notin \xi \text{ for all }y\in Y\}.
\]
These sets are open and closed and generate the subspace topology on $\Ef(S)$ as $X$ and $Y$ range over all finite subsets of $E(S)$. Let $\Eu(S)$ denote the subspace of ultrafilters. We shall denote by $\Et(S)$ the closure of $\Eu(S)$ in $\Ef(S)$ and call this the space of {\em tight} filters. 

If $X$ is a topological space and $S$ is an inverse semigroup, recall that an {\em action} of $S$ on $X$ is a pair $(\{D_e\}_{e\in E(S)},\{\theta_s\}_{s\in S})$ such that each $D_e\subset X$ is open, the union of the $D_e$ coincides with $X$, each map $\theta_s: D_{s^*s}\to D_{ss^*}$ is continuous and bijective, and for all $s, t\in S$ we have $\theta_s\circ \theta_t = \theta_{st}$, where composition is on the largest domain possible. These properties imply that $\theta_{s^*} = \theta_s^{-1}$ and so each $\theta_s$ is actually a homeomorphism. 

There is a canonical way to construct an \'etale groupoid from an inverse semigroup action. Let $\theta = (\{D_e\}_{e\in E(S)},\{\theta_s\}_{s\in S})$ be an action of an inverse semigroup $S$ on a space $X$, and let
\[
S\times_\theta X : = \{ (s, x)\in S\times X\mid x\in D_{s^*s}\}.
\]
For $(s, x), (t, y)\in S\times_\theta X$ we write $(s,x)\sim (t, y)$ if $x = y$ and there exists $e\in E(S)$ such that $x\in D_e$ and $se = te$. It is straightforward to check that $\sim$ is an equivalence relation. We write $[s, x]$ for the equivalence class of $(s, x)$ and we let $\G(S, X, \theta)$ denote the set of all such equivalence classes. This set becomes a groupoid when given the operations
\[
[s,x]^{-1} = [s^*, \theta_s(x)], \hspace{0.5cm} r([s,x]) = \theta_s(x), \hspace{0.5cm} d([s,x]) = x, \hspace{0.5cm} [t, \theta_s(x)][s,x] = [ts, x].
\]
For $s\in S$ and $U$ an open subset of $D_{s^*s}$, let
\[
\Theta(s, U) =\{[s, x]\mid x\in U\}.
\]
As $s$ and $U$ vary, these sets form a basis for an \'etale topology on $\G(S, X, \theta)$;  with this topology $\G(S, X, \theta)$ is called the {\em groupoid of germs} of the action $\theta$. In this topology, $\Theta(s, U)$ is an open bisection, and if $U$ is in addition closed (resp. compact), $\Theta(s, U)$ is a clopen (resp. compact open) bisection. It is easy to see that the orbit of a point $x\in X$ under the groupoid of germs is the set $\{\theta_s(x)\mid s\in S\}$.
 
An inverse semigroup acts naturally on $\Et(S)$. Let $D_e = \{\xi\in \Et(S)\mid e\in \xi\}$, and define $\theta_s: D_{s^*s}\to D_{ss^*}$ by
\[
\theta_s(\xi) = \overline{s\xi s^*} = \overline{\{ses^*\mid e\in \xi\}}.
\]
The groupoid of germs of this action is denoted
\[
\Gt(S): = \G(S, \Et, \theta)
\]
and is called the {\em tight groupoid} of $S$. The C*-algebra of $\Gt(S)$ is naturally isomorphic to $\Ct(S)$; in particular the map $\pi:S \to C^*(\Gt)$ given by $\pi(s) = \chi_{\Theta(s, D_{s^*s})}$ is a tight representation.

If $P$ is a right LCM semigroup and $\s$ is as in \eqref{Sdef}, then it is easy to see that $\J(P)$ and $E(\s)$ are isomorphic as semilattices, with the isomorphism being the map $pP\mapsto [p,p]$. The spaces of filters and ultrafilters in $\J(P)$ were considered in previous study of C*-algebras associated to $P$, though filters were termed {\em directed} and {\em hereditary} subsets of $\J(P)$ while the ultrafilters were called {\em maximal} directed hereditary subsets. In what follows, we will consider the elements of $\Et(\s)$ as tight filters in $\J(P)$, and will shorten $D_{[p,p]}$ to $D_p$ (noting that $D_p = D_{pu}$ for all $u\in U(P)$. We will then have, 
\[
\theta_{[p,q]}: D_q\to D_p
\]
\[
\theta_{[p,q]}(\xi) = \{p(q^{-1}rP)\mid rP\in \xi\},
\]
for any $p, q\in P$ and $\xi\in \Et(\s)$ with $qP\in \xi$.

\subsection{Simplicity of $\Q(P)$}\label{simplicitysubsection}
We now characterize when $\Q(P)$ is simple using the fact that it is isomorphic to the C*-algebra of the \'etale groupoid $\Gt(\s)$. To use the characterization of Theorem \ref{groupoidsimple} (which is \cite[Theorem 5.1]{BCFS14}), we need conditions which guarantee that $\Gt(\s)$ is Hausdorff, minimal, and topologically principal. We begin with Hausdorff.

As noted in \cite[\S 6]{Ex08}, $\Gt(S)$ is Hausdorff if $S$ is {\em E*-unitary}, that is, for all $s\in S$ and $e\in E(S)\setminus \{0\}$, $e \leqslant s$ implies that $s\in E(S)$. Norling notes in \cite[Corollary 3.24]{No14} that if $P$ is a right LCM semigroup with identity and $\s$ is as in \eqref{Sdef}, then $P$ is cancellative if and only if $\s$ is E*-unitary. Thus if $P$ is cancellative, $\Gt(\s)$ is Hausdorff. However, by \cite[Theorem 3.16]{EP14} we can do a little bit better. We are more precise below, but what we prove is essentially that $\Gt(\s)$ is Hausdorff if and only if the counterexamples to right cancellativity in $P$ have a ``finite cover'' in some sense.
 
For $p, q\in P$, let $P_{p, q} = \{b\in P\mid pb = qb\}$. If $P_{p, q}$ is nonempty then it is a right ideal of $P$, and in this case we say that $p$ and $q$ {\em meet}. We introduce the following condition that $P$ may satisfy.
\begin{enumerate}
\item[(H)]For every $p, q\in P$ which meet, there is a finite set $F\subset P$ such that $pf = qf$ for all $f\in F$ and whenever we have $b\in P$ such that $pb = qb$, there is an $f\in F$ such that $fP\cap bP \neq \EM$. 
\end{enumerate}
One sees that (H) is weaker than right cancellativity, since if $P$ is right cancellative $p$ only meets $q$ when $p=q$, and in this case $P_{p, q} = P$ and the finite set $F = \{1_P\}$ verifies (H).
\begin{prop}\label{hausdorff}
Let $P$ be a right LCM semigroup with identity, and let $\s$ be as in \eqref{Sdef}. Then $\Gt(\s)$ is Hausdorff if and only if $P$ satisfies condition (H).
\end{prop}
\begin{proof}
We shall show that condition (H) is equivalent to the set
\[
J_{[p,q]} = \{rP\mid[r,r]\leqslant [p,q]\}
\]
either being empty or having a finite cover for all $p,q\in P$. If we do this, then the conclusion will follow from \cite[Theorem 3.16]{EP14}. Notice that if we have $p, q, r\in P$ such that $[r,r]\leqslant [p,q]$, then $[p,q][r,r] = [r,r]$. If $rP \cap qP = kP$ and $rr'=qq' = k$, then we obtain that $[pq', rr'] = [r,r]$ implying that $r'\in U(P)$. This means that $rP = kP$, and so we may assume (perhaps by rechoosing $q'$) that $[pq', r] = [r,r]$, and so $pq' = r= qq'$. Thus, for each element $rP\in J_{[p,q]}$, there is an element $p_r:= q'$ such that $pp_r = r = qp_r$.

First, assume that for all $p, q\in P$ the set $J_{[p,q]}$ is empty or has a finite cover. Suppose that $p, q\in P$ meet, that is, there exists $b\in P$ such that $pb = qb$. Then  
\[
[p,q][qb,qb] = [pb,qb] = [qb, qb] 
\]
and so $qbP= pbP\in J_{[p,q]}$, meaning that $ J_{[p,q]}$ is not empty. Hence there is a finite set $F\subset P$ such that $fP\in  J_{[p,q]}$ for all $f\in F$ and for all $rP\in  J_{[p,q]}$ there exists $f\in F$ such that $rP\cap fP \neq \EM$. By the above, there exists $p_f\in P$ such that $pp_f = f = qp_f$. We now see that the finite set $\{p_f\}_{f\in F}$ verifies (H), because if we have $d$ such that $pd = qd$, there is $f\in F$ such that $fP\cap pdP \neq \EM$, which implies that $p_fP\cap dP \neq \EM$.

Conversely, suppose $P$ satisfies condition (H). If $p, q$ do not meet, then the above discussion shows that $J_{[p,q]}$ is empty. If $p,q$ do meet, let $F$ be the finite set guaranteed by (H), and consider the finite set
\[
pF = qF = \{pf\mid f\in F\}.
\]
If $rP\in J_{[p,q]}$, then again by the above there exists $r'$ such that $pr' = qr' = r$. So, there exists $f\in F$ such that $fP\cap r'P \neq \EM$, which implies that $pfP\cap pr'P = pfP\cap rP\neq \EM$ and we are done.
\end{proof}

Now that we have addressed when $\Gt(\s)$ is Hausdorff, we turn to minimality.
\begin{lem}\label{minimallem}Let $P$ be a right LCM semigroup with identity, and let $\s$ be as in \eqref{Sdef}. Then $\Gt(\s)$ is minimal.
\end{lem}
\begin{proof}
We will show that for every ultrafilter $\xi$ and open set $U(X,Y) \subset \Et(\s)$, there is a $[p, q]\in\s$ such that $\theta_{[p,q]}(\xi)\in U(X,Y)$. The set $U(X,Y)$ is open, so it contains an ultrafilter $\eta$ which must have the property that $xP\in \eta$ for all $x\in X$ and, for all $y\in Y$, there exists $p_y\in P$ such that $p_yP\in \eta$ and $p_yP\cap yP = \EM$. Because $\eta$ is closed under intersection, there must be $r\in P$ such that
\[
\left(\bigcap_{x\in X}xP \right)\bigcap \left(\bigcap_{y\in Y}p_yP\right) = rP.
\]
Now, $\zeta: = \theta_{[r, 1_P]}(\xi)$ is an ultrafilter which contains $rP$. Since $\zeta$ is upwards directed, $xP, p_yP\in \zeta$ for all $x\in X, y\in Y$, and so $\zeta\in U(X,Y)$. Thus, the orbit of every ultrafilter is dense. Each open set contains an ultrafilter, so the only nonempty open invariant subset of $\Et$ is $\Et$, and so $\G(\s, \Et, \theta)$ is minimal.  
\end{proof}

Lastly, we discuss conditions which guarantee that $\Gt(\s)$ is topologically principal. We use the following concepts from \cite{EP14}. For an action $(\{D_e\}_{e\in E(S)}, \{\alpha_s\}_{s\in S})$ of an inverse semigroup $S$ on a locally compact Hausdorff space $X$, and $s\in S$, let
\[
F_s = \{ x\in X\mid \alpha_s(x) = x\}
\]
and call this the set of {\em fixed points} for $s$. Also let
\begin{eqnarray}
TF_s &=& \{ x\in X \mid \text{there exists } e\in E(S)\text{ such that } 0\neq e \leqslant s\text{ and }x\in D_e\}\nonumber\\
     &=& \bigcup_{e\leqslant s}D_e \label{TFsopen}
\end{eqnarray}
and call this the set of {\em trivially fixed points} for $s$. From \cite[Theorem 3.15]{EP14}, the groupoid of germs $\G(S, X, \alpha)$ is Hausdorff if and only if $TF_s$ is closed in $D_{s^*s}$ for all $s\in S$. 

\begin{defn}
An action  $(\{D_e\}_{e\in E(S)}, \{\alpha_s\}_{s\in S})$ of an inverse semigroup $S$ on a locally compact Hausdorff space $X$ is said to be {\em topologically free} if the interior of $F_s$ is contained in $TF_s$ for all $s\in S$.
\end{defn}

We note that if $x$ is trivially fixed by some $s$ with $e\leqslant s$ and $x\in D_e$, we have $\alpha_s(x) = \alpha_s(\alpha_e(x)) = \alpha_{se}(x) = \alpha_e(x) = x$, so $x$ is fixed, that is to say that
\[
TF_s\subset F_s \hspace{1cm} \text{for all } s\in S.
\] 
Also, by \eqref{TFsopen}, $TF_s$ is open and so is contained in the interior of $F_s$. Hence stating that $\alpha$ is topologically free is equivalent to saying that $TF_s = \mathring{F}_s$

\begin{theo}\cite[Theorem 4.7]{EP14}
An action $(\{D_e\}_{e\in E(S)}, \{\alpha_s\}_{s\in S})$ of an inverse semigroup $S$ on a locally compact Hausdorff space $X$ is topologically free if and only if $\G(S, X, \alpha)$ is essentially principal.
\end{theo}

We now show that we can characterize when the canonical action of $\s$ on $\Et(\s)$ is topologically free by considering the behaviour of a subsemigroup of $P$ which generalizes one originally considered in \cite{CL07}. 
\begin{prop}\label{coreprop}
Let $P$ be a right LCM semigroup with identity. Then the set 
\begin{equation}\label{core}
P_0 := \{p\in P\mid pP\cap qP\neq \EM\text{ for all }q\in P\}
\end{equation}
is a subsemigroup of $P$ which contains the identity. Furthermore, 
\begin{enumerate}
\item $pq\in P_0$ implies that $p, q\in P_0$, and
\item$p, q\in P_0$ and $pP\cap qP = rP$ implies that $r\in P_0$. 
\end{enumerate}
\end{prop}
\begin{proof}
The details of this proof are almost identical to that of \cite[Lemma 5.3]{CL07}. For instance, take $p, q\in P_0$, and $r\in P$. A short calculation shows that $pqP\cap rP = p(qP\cap p^{-1}(pP\cap rP))$, and since $p, q\in P_0$ this must be nonempty, whence $pq\in P_0$.
\end{proof}
\begin{defn}\label{coredef}\cite[Definition 5.4]{CL07} Let $P$ be a right LCM semigroup with identity. Then the subsemigroup $P_0\subset P$ from \eqref{core} is called the {\em core} of $P$.
\end{defn}
The subsemigroup $P_0$ was defined in $\cite{CL07}$ when $P$ is quasi-lattice ordered, though it still makes sense in our context. We note that for all $p\in P_0$, the singleton $\{p\}$ is a foundation set, and so $v_p$ is a unitary in $\Q(P)$. We also note that $U(P)\subset P_0$, though this inclusion may be proper.\footnote{We mention in passing that a cancellative semigroup $P$ for which $P_0 = P$ is called an {\em Ore semigroup}. We do not know if this a reason for the terminology of \cite{CL07}, but in light of this perhaps we should call $P_0$ the cOre of $P$.}

Now let
\[
\s_0 := \{[p,q]\in \s\mid p,q\in P_0\}.
\]
We will also call this the {\em core} of $\s$. For $a, b, c, d\in P_0$, we see that $[a,b][c,d] = [ab', dc']$ where $bP\cap cP = rP$ and $bb' = cc' = r$ implies that $r\in P_0$ and hence $b', c'\in P_0$. Thus, $\s_0$ is an inverse subsemigroup of $\s$. We note that $\s_0$ does not contain the zero element.

\begin{prop}\label{EPcore}\footnote{In preparation of this work we noticed that this proposition may be proved for more general inverse semigroup actions, see Appendix \ref{coreappendix}.}
Let $P$ be a right LCM semigroup with identity which satisfies condition (H), and let $\s$ be as in \eqref{Sdef}. Then $\Gt(\s)$ is essentially principal if and only if for all $s\in \s_0$, each interior fixed point of $s$ is trivially fixed.
\end{prop}
\begin{proof}
The ``only if'' direction is obvious. So, assume that for each $s\in \s_0$, each interior fixed point of $s$ is trivially fixed, and suppose that $\Gt(\s)$ is not essentially principal. Then there must exist $[p,q]\in \s$ such that $TF_{[p,q]}\subsetneq \mathring{F}_{[p,q]}$. Since $P$ satisfies condition (H), $TF_{[p,q]}$ is closed in $D_{q}$ and so $\mathring{F}_{[p,q]}\setminus TF_{[p,q]}$ is open, so we can find an open set $U \subset \mathring{F}_{[p,q]}\setminus TF_{[p,q]}$. Since $U$ is open, it must contain an ultrafilter $\xi\in U$. If we are able to find a $b\in P$ such that $bP\in \xi$ and $[1_P, b][p,q][b,1_P]\in \s_0$, then $D_b$ would contain $\xi$, and so $\theta_{[1_P,b]}(\xi)$ is fixed by $[1_P, b][p,q][b,1_P]$. By assumption, it must be in $TF_{[1_P, b][p,q][b,1_P]}$, so we can find a nonzero idempotent $[r,r]$ such that $[r,r]\leqslant [1_P, b][p,q][b,1_P]$ with $rP\in \theta_{[1_P,b]}(\xi)$. A short calculation shows that this implies that $[b,1_P][r,r][1_P, b]= [br, br]$ is a nonzero idempotent less than $[p,q]$. Furthermore, $rP\in \theta_{[1_P,b]}(\xi)$ implies that $brP\in \xi$, and so $\xi\in D_{br}$ which would imply that $\xi$ is trivially fixed by $[p,q]$, a contradiction. So finding such a $b\in P$ would prove the proposition.

So, suppose that $[1_P, b][p,q][b,1_P]\notin \s_0$ for all $b\in P$ such that $bP\in \xi$, and fix an element $bP\in \xi$. Because $\xi$ is fixed by $[p,q]$, we have that $p(q^{-1}(bP))\in \xi$. Hence, there exist $b_1, q_1, r_1\in P$ such that $bP\cap qP = r_1P$ and $bb_1 = qq_1 = r_1$, and so we have $p(q^{-1}(bP) = pq_1P\in \xi$. Since $bP, pq_1P\in \xi$, there exist $p_1, b_2, r_2\in P$ such that $pq_1P\cap bP = r_2P$ and $pq_1p_1 = bb_2 = r_2$. Upon redefining $a:= b_2$ and $c:= b_1p_1$, a short calculation shows that 
\[
[1_P,b][p,q][b, 1_P] = [a,c].
\] 
Since we are assuming that this is not an element of $\s_0$, we must have that one of $a$ or $c$ is not an element of $P_0$. Suppose for the moment that $a\notin P_0$, which means there exists $z\in P$ such that $zP\cap aP = \EM$. Letting $\xi_{bP} = \theta_{[bz, 1]}(\xi)$, we see that $bP\in \xi_{bP}$ (because $bzP\in \xi_{bP}$ and $\xi_{bP}$ is upwards closed). However, $\xi_{bP}$ is not fixed by $[p,q]$ (whether $\theta_{[p,q]}(\xi_{bP})$ is defined or not) because any filter containing $bP$ and fixed by $[p,q]$ must contain $r_2P = baP$ by the same reasoning as above, and since $\xi_{bP}$ is closed under intersections this would mean that it contains $baP\cap bzP$, which is empty by assumption. In a similar fashion, we can construct an ultrafilter $\xi_{bP}$ containing $bP$ not fixed by $[p, q]$ if we instead assume that $c\notin P_0$. Hence we have constructed a net $\{\xi_b\}_{bP\in \xi}$ of ultrafilters each of which contains $bP$ but none of which is fixed by $[p,q]$. This net converges to $\xi$, and so $U$ contains a point in this net, which is impossible since $U$ is fixed by $[p,q]$. Hence, we are forced to conclude that we can find $b\in P$ such that $bP\in \xi$ and $[1_P, b][p,q][b,1_P]\in \s_0$, and so we are done.
\end{proof}

We would like an algebraic condition on $P$ which guarantees $\Gt(\s)$ is essentially principal. To do this we recall some terminology from \cite{EP14}

\begin{defn}
Let $S$ be an inverse semigroup, let $s\in S$. If $e\in E(S)$ is an idempotent with $e\leqslant s^*s$ it is said that
\begin{enumerate}
\item  $e$ is {\em fixed} by $s$ if $e\leqslant s$ (ie, $se = e$), and
\item  $e$ is {\em weakly fixed} by $s$ if $sfs^*f \neq 0$ for every nonzero idempotent $f\leqslant e$. 
\end{enumerate} 
\end{defn}
We translate this terminology to our situation in the following lemma.
\begin{lem}
Let $P$ be a right LCM semigroup with identity, and take $p, q, r\in P$ such that $r = qk$ for some $k\in P$. Then 
\begin{enumerate}
\item $[r,r] = [qk, qk]$ is fixed by $[p,q]$ if and only if $pk=r=qk$, and
\item $[r,r] = [qk,qk]$ is weakly fixed by $[p,q]$ if and only if for all $a\in P$;
\[
qkaP\cap pkaP\neq \EM;
\]
\end{enumerate} 
Also, if $[r,r]$ is fixed by $[p,q]$, it is weakly fixed by $[p,q]$.
\end{lem}
\begin{proof}
Note that in the statement of the lemma, we only consider $r$'s which are right multiples of $q$ because this is exactly what it means to have $[r,r]\leqslant [p,q]^*[p,q]$. 
\begin{enumerate}
\item If $[r,r]$ is fixed by $[p,q]$, we have $[p,q][r,r] = [r,r]$ and so $[r,r] = [pk, r]$. Thus $pk = r = qk$. On the other hand, if $pk = qk = r$, then $[p,q][r,r] = [p,q][qk, qk]= [pk,qk] = [r,r]$, and so $[r,r]$ is fixed by $[p,q]$.
\item Suppose that $[r,r]$ is weakly fixed by $[p,q]$. An idempotent is below $[r,r]$ if and only if it is of the form $[ra, ra]$. Thus $[r,r]$ being weakly fixed by $[p,q]$ implies that for every $a\in P$, 
\[
0\neq [p,q][ra,ra][q,p][ra,ra] = [p,q][qka,qka][q,p][ra,ra] = [pka, pka][ra,ra].
\]
Hence $raP\cap pkaP\neq \EM$. Conversely, if $raP\cap pkaP\neq \EM$ for all $a\in P$, then with the same calculation above we see that for all $a\in P$ the product $[p,q][ra,ra][q,p][ra,ra] \neq 0$, and so $[r,r]$ is weakly fixed by $[p,q]$. 
\end{enumerate}
As in the general situation, it is clear that each fixed idempotent is weakly fixed.
\end{proof}
 
The following statement is implicit in the proof of \cite[Theorem 4.10]{EP14}, though we spell it out here for emphasis.
\begin{lem}
Let $S$ be an inverse semigroup, and suppose that either 
\begin{enumerate}
\item[(i)]every tight filter in $E(S)$ is an ultrafilter, or
\item[(ii)]for every $s\in S$, the set $J_s = \{e\in E(S)\mid e\leqslant s\}$ has a finite cover.
\end{enumerate}
Then for each $s\in S$, $\mathring{F}_s \subset TF_s$ if and only if for all $e$ weakly fixed by $s$, there is a finite cover for $J_e$ consisting of fixed idempotents.
\end{lem}
The following is a rephrasing of the above result for our situation.
\begin{lem}\label{EPalgebraic}
Let $P$ be a right LCM semigroup with identity which satisfies condition (H), or such that the only tight filters in $\J(P)$ are ultrafilters. Then $\mathring{F}_{[p,q]}\subset TF_{[p,q]}$ if and only if $[p,q]$ satisfies the following condition:
\begin{enumerate}
\item[(EP)] for all $[qk, qk]$ weakly fixed by $[p,q]$, there exists a foundation set $F\subset P$ such that $qkf = pkf$ for all $f\in F$.
\end{enumerate}
\end{lem}
We note that any idempotent $[p,p]$ satisfies (EP) using the foundation set $\{1_P\}$.

One notices that we still fall slightly short of being able to apply Theorem \ref{groupoidsimple}, because we have only given conditions under which $\Gt(\s)$ is essentially principal, not topologically principal. As discussed above Theorem \ref{groupoidsimple}, these two notions are equivalent when $\Gt(\s)$ is Hausdorff and second countable. We are considering only countable semigroups $P$, and one easily sees that this guarantees that $\Gt(\s)$ is second countable.

We now come to the main result of this section.

\begin{theo}\label{Qsimple}
Let $P$ be a right LCM semigroup with identity which satisfies condition (H), let $P_0$ be the core of $P$, and let $\s$ be as in \eqref{Sdef}. Then $\Q(P)$ is simple if and only if
\begin{enumerate}
\item $\Q(P)\cong C^*_r(\Gt(\s))$, and
\item for all $p, q\in P_0$, the element $[p,q]$ satisfies condition (EP).
\end{enumerate}
\end{theo}
\begin{proof}
By Proposition \ref{hausdorff}, $\Gt(\s)$ is Hausdorff, so we can apply Theorem \ref{groupoidsimple}. By Lemma \ref{minimallem}, $\Gt(\s)$ is always minimal. By  \cite[Theorem 4.7]{EP14}, Proposition \ref{EPcore}, and Lemma \ref{EPalgebraic}, $\Gt(\s)$ is topologically principal if and only if we have (2) above. The result follows. 
\end{proof}

\subsection{Pure infiniteness of $\Q(P)$}\label{purelyinfinitesubsection}
We recall a definition from \cite{EP14}.
\begin{defn}\label{locallycontractingsemigroup}
An inverse semigroup $S$ is called {\em locally contracting} if for every nonzero $e\in E(S)$ there exists $s\in S$ and a finite set $F = \{f_0, f_1, \dots f_n\}\subset E(S)\setminus \{0\}$  with $n\geq 0$ such that for all $0\leq i \leq 1$ we have
\begin{enumerate}
\item $f_i\leq es^*s$,
\item there exists $f\in F$ such that $sf_is^*f \neq 0$, and
\item $f_0sf_i = 0$.
\end{enumerate} 
\end{defn}
As one might guess from the name, if $S$ is locally contracting then $\Gt(S)$ is locally contracting by \cite[Corollary 6.6]{EP14}.

\begin{lem}\label{lemlocallycontracting}
Let $P$ be a right LCM semigroup with identity and let $\s$ be as in \eqref{Sdef}. Then $\s$ is locally contracting if and only if $P\neq P_0$.
\end{lem}
\begin{proof}
The ``only if'' direction is trivial, because if $P = P_0$ we could not satisfy part 3 of Definition \ref{locallycontractingsemigroup}, as the product of the two idempotents $f_0$ and $sf_is^*$ could not be zero.

For the ``if'' direction, we suppose that $P\neq P_0$, and hence we can find $p, q\in P$ such that $pP\cap qP = \EM$. By \cite[Proposition 6.7]{EP14}, we will be done if for every $r\in P$ we can find $a\in P$ and $f_0, f_1\in P$ such that $f_0P\subset f_1P\subset rP$, $af_1P\subset f_1P$ and $[f_0,f_0][a,1_P][f_1, f_1] = 0$. To this end, let
\[
a = f_1 = rp \hspace{1cm} f_0 = rprq.
\]
Then clearly $f_0P\subset f_1P\subset rP$, and $af_1P = rprpP\subset f_1P$. We also have that
\[
[f_0,f_0][a,1_P][f_1, f_1] = [rprq, rprq][rp, 1_P][rp, rp] = [rprq, rprq][rprp, rp]
\]
and since $rprqP\cap rprpP = \EM$ by assumption, this product is zero. 
\end{proof}
\begin{theo}\label{purelyinfinite}
Let $P$ be a right LCM semigroup with identity which satisfies condition (H), and suppose that $\Q(P)$ is simple. Then $\Q(P)$ is purely infinite if and only if $\Gt(\s)$ is not the trivial (one-point) groupoid.
\end{theo}
\begin{proof}
The ``only if'' direction is clear, because if $\Gt(\s)$ is one point, its C*-algebra is isomorphic to $\mathbb{C}$, which is not purely infinite.

On the other hand, if $\Gt(\s)$ is not the one-point groupoid, we have two cases. If $\Et(\s)$ is one point then there are no points with trivial isotropy, and so $\Gt(\s)$ is not topologically principal, contradicting Theorem \ref{groupoidsimple}. If $\Et(\s)$ has more than one point, then there are at least two distinct ultrafilters in $\J(P)$. Hence we can find a $p\in P$ and an ultrafilter $\xi$ such that $pP\notin \xi$, and since $\xi$ is an ultrafilter there must be $qP\in \xi$ such that $pP\cap qP = \EM$. Thus neither $p$ nor $q$ is in $P_0$, and so $P\neq P_0$ implying that $\s$ is locally contracting by Lemma \ref{lemlocallycontracting}. Hence, by \cite[Corollary 6.6]{EP14} and \cite[Proposition 2.4]{AD97} $C^*_r(\Gt(\s))\cong \Q(P)$ is purely infinite.
\end{proof}
Hence, in the presence of simplicity, pure infiniteness of $\Q(P)$ follows automatically in all but the most trivial cases.

In \cite[Theorem 5.3]{BL14}, the authors give conditions under which $C^*(P)$ is simple and purely infinite for a right LCM semigroup $P$. Since $\Q(P)$ is in some sense the smallest quotient of $C^*(P)$, it is not surprising that $\Q(P)$ is simple under these milder conditions.
\section{Examples}\label{examplessection}
\subsection{Free semigroups}\label{freesemigroupssection}
Let $X$ be a finite set, and let $X^n$ denote the set of words of length $n$ in $X$, with $X^0$ consisting of a single empty word, $\emptyset$. Let 
\[
X^* = \bigcup_{n\geq 0}X^n.
\]
Then $X^*$ becomes a semigroup with the operation of concatenation: if $\alpha = \alpha_1\alpha_2\cdots\alpha_{k}$ and $\beta = \beta_1\beta_2\cdots\beta_l$ then their product is $\alpha\beta = \alpha_1\alpha_2\cdots\alpha_{k}\beta_1\beta_2\cdots\beta_l$, while the empty word is the identity. If $\alpha\in X^n$, we write $|\alpha| = n$ and say that the {\em length} of $\alpha$ is $n$.  The core of this semigroup is $X^*_0 = U(X^*) = \{\emptyset\}$. If we have $\alpha, \beta\in X^*$, then $\alpha X^* = \beta X^*$ if and only if $\alpha = \beta$. Furthermore, $X^*$ is left cancellative, and either $\alpha X^* \cap \beta X^* = \EM$ or one is included in the other, so $X^*$ is right LCM. 

From the relations (L1)-(L4) it follows easily that $C^*(X^*)$ is the universal unital C*-algebra generated by isometries $v_1, \dots, v_{|X|}$  such that 
\[
v_i^*v_j = \begin{cases}1 &\text{if }i = j\\0&\text{ otherwise,}\end{cases}
\] 
that is, $C^*(X^*)$ is isomorphic to the Toeplitz algebra $\mathcal{TO}_{|X|}$. Furthermore, the set $X = X^1\subset X^*$ is a foundation set, and so in $\Q(X^*)$ we have
\[
0 = \prod_{x\in X}(1 - e_{xX^*}) = 1 - \bigvee_{x\in X}v_xv_x^* = 1-\sum_{x\in X}v_xv_x^*
\]
\[
\Rightarrow \sum_{x\in X}v_xv_x^* = 1
\]
and since the Cuntz algebra $\mathcal{O}_{|X|}$ is the universal C*-algebra generated by such elements, there is a surjective $*$-homomorphism from $\mathcal{O}_{|X|}$ to $\Q(X^*)$ which must be an isomorphism because $\mathcal{O}_{|X|}$ is simple. 

Principal left ideals of $X^*$ are either disjoint or comparable by inclusion, and hence ultrafilters are maximal well-ordered subsets of $\J(X^*)$. The space of ultrafilters can be identified with the compact space $\Sigma_X$ of right-infinite words in $X$ via the homeomorphism
\[
\alpha\in \Sigma_X \mapsto \{X^*, \alpha_1X^*, \alpha_1\alpha_2X^*, \alpha_1\alpha_2\alpha_3X^*, \dots\} \in \Et(\s).
\]
Here every tight filter is an ultrafilter. Because $X^*$ is right cancellative (in fact, it can be embedded in the free group on $|X|$ elements) it satisfies condition (H). The inverse semigroup $\s$ from \eqref{Sdef} is known in the literature as the {\em polycyclic monoid} on $|X|$ generators. For all $\alpha\in X^*$, the idempotent $[\alpha, \alpha]$ is weakly fixed by $[\emptyset, \emptyset]$, and $[\emptyset, \emptyset]$ trivially satisfies condition (EP). 

\subsection{Right LCM semigroups embedded in groups}

In \cite{Li13} Li builds upon his earlier work and makes a comprehensive study of the C*-algebras of semigroups which may be embedded into groups. There the semilattice of constructible ideals $\J(P)$ is considered, though it is not always equal to the set of all principal right ideals. There, the set of filters is denoted $\Sigma$, the set of ultrafilters is denoted $\Sigma_{\text{max}}$ and its closure is denoted $\partial\Sigma$. In \cite[\S 5.2]{Li13} an inverse semigroup analogous to our $\s$ from \eqref{Sdef} is defined; this inverse semigroup acts on $\Sigma$. The groupoid of germs of this action is the universal groupoid for $\s$, and the C*-algebra of its reduction to $\partial\Sigma$ (that is to say, $\Gt(\s)$), is identified as a suitable boundary quotient. Our only contribution to the literature for this situation would seem to be the isomorphism between this boundary quotient and the one defined in \cite{BRRW14}. We do note that our result Proposition \ref{EPcore} is analogous to \cite[Proposition 7.20]{Li13}, and indeed it seems both were inspired by \cite[Proposition 5.5]{CL07}.

\subsection{Zappa-Sz\'ep products of semigroups}\label{zappaszepsection}

The following is a construction considered in \cite{BRRW14}. Let $U$ and $A$ be  semigroups with identities $1_U$ and $1_A$ respectively and suppose there exist maps $A\times U \to U$ given by $(a, u)\mapsto a\cdot u$, and $A\times U \to A$ given by $(a,u)\mapsto \left.a\right|_u$ which satisfy

\vspace{-0.3cm}\begin{tabular}{p{7cm} p{7cm}}
\begin{enumerate}\addtolength{\itemsep}{-0.5\baselineskip}
\item[(ZS1)] $1_A\cdot u = u$
\item[(ZS2)] $(ab)\cdot u = a\cdot(b\cdot u)$
\item[(ZS3)] $a\cdot 1_U = 1_U$
\item[(ZS4)] $a\cdot(uv) = (a\cdot u)(\left.a\right|_u\cdot v)$
\end{enumerate}
&
\begin{enumerate}\addtolength{\itemsep}{-0.5\baselineskip}
\item[(ZS5)] $\left.a\right|_{1_U} = a$
\item[(ZS6)] $\left.a\right|_{uv} = \left.a\right|_u\left.\hspace{-0.1cm}\right|_v$
\item[(ZS7)] $\left.1_A\right|_u = 1_A$
\item[(ZS8)] $\left.ab\right|_u = \left.a\right|_{b\cdot u}\left.b\right|_u$
\end{enumerate}
\end{tabular}

\noindent for all $u, v\in U$ and $a, b\in A$. Then $U\times A$ becomes a semigroup with identity $(1_U, 1_A)$ when given the operation
\[
(u,a)(v,b) = (u(a\cdot v), \left.a\right|_vb).
\]
This is called the {\em Zappa-Sz\'ep product} of $U$ and $A$, and is denoted $U\bowtie A$. If in addition to the above, we have that
\begin{enumerate}\addtolength{\itemsep}{-0.5\baselineskip}
\item[(i)]$U$ and $A$ are both left cancellative,
\item[(ii)]$U$ is right LCM,
\item[(iii)]$\J(A)$ is totally ordered by inclusion, and
\item[(iv)]the map $u\mapsto a\cdot u$ is a bijection on $U$ for each $a\in A$,
\end{enumerate}
then $U\bowtie A$ is a right LCM semigroup as well, see \cite[Lemma 3.3]{BRRW14}. 

By Theorem \ref{maintheorem}, the boundary quotient $\Q(U\bowtie A)$ defined in \cite[Definition 5.1]{BRRW14} is isomorphic to the C*-algebra of an \'etale groupoid $\Gt(\s)$ (where $\s$ is as in \eqref{Sdef}) whose unit space is homeomorphic to the space of tight filters in $\J(U\bowtie A)$. To use Theorems \ref{Qsimple} and \ref{purelyinfinite} requires that we know the nature of the core of our semigroup, and in this case the core has an easily describable form. Firstly, by \cite[Lemma 5.3(a)]{BRRW14}, each element $(1_U,a)$ is in the core of $U\bowtie A$. Furthermore, by \cite[Remark 3.4]{BRRW14}, for $u, v\in U$ and $a, b\in A$, we have
\[
(u, a)U\bowtie A \cap (v, b)U\bowtie A = \EM \hspace{0.5cm}\Leftrightarrow\hspace{0.5cm} uU\cap vU = \EM.
\]
Therefore, $\{(u,a)\}$ is a one-point foundation set in $U\bowtie A$ if and only if $\{u\}$ is a foundation set in $U$. Hence the core of $U\bowtie A$ is
\[
(U\bowtie A)_0 = \{(u, a)\in U\bowtie A\mid u\in U_0\}.
\]
By Proposition \ref{coreprop} this is a subsemigroup of $U\bowtie A$, and in particular, we have that for all $u\in U_0$, $a\cdot u \in U_0$ for all $a\in A$. Thus we are justified writing $(U\bowtie A)_0 = U_0\bowtie A$. Without having more information about $U$ and $A$ we cannot say much more, though in the sequel we consider a specific example for which we can.

\subsection{Self-similar groups}

We close with an example which is a specific case of the situation from \S \ref{zappaszepsection}. The conclusions we come to in this section are known, and combine the results of \cite{EPSep14} and \cite{BRRW14}. Indeed, generalizing the results implicit in combining \cite{EP13} (which was a preliminary version of \cite{EPSep14}) and \cite{BRRW14} was a major inspiration for this work. We present what follows to illustrate our results in the context of this interesting example.

Let $X$ be a finite set, let $G$ be a group, and let $X^*$ be as in \S \ref{freesemigroupssection}. Suppose that we have a length-preserving action of $G$ on $X^*$, with $(g, \alpha)\mapsto g\cdot \alpha$, such that for all $g\in G$, $x\in X$ there exists a unique element of $G$, denoted $\left.g\right|_x$, such that for all $\alpha\in X^*$
\[
g(x\alpha) = (g\cdot x)(\left.g\right|_x\cdot \alpha).
\]
In this case, the pair $(G,X)$ is called a {\em self-similar group}. In \cite{Nek09}, Nekrashevych associates a C*-algebra to $(G, X)$, denoted $\mathcal{O}_{G, X}$, which is the universal C*-algebra generated by a set of isometries $\{s_x\}_{x\in X}$ and a unitary representation $\{u_g\}_{g\in G}$ satisfying
\begin{enumerate}\addtolength{\itemsep}{-0.5\baselineskip}
\item[(i)]$s_x^*s_y = 0$ if $x\neq y$, 
\item[(ii)]$\sum_{x\in X}s_xs_x^* = 1$,
\item[(iii)]$u_gs_x = s_{g\cdot x}u_{\left.g\right|_x}$.
\end{enumerate} 

If one defines, for $\alpha\in X^*$ and $g\in G$,
\[
\left.g\right|_\alpha := \left.g\right|_{\alpha_1}\hspace{-0.1cm}\left.\right|_{\alpha_2}\cdots \hspace{-0.1cm}\left.\right|_{\alpha_{|\alpha|}}
\]
then the free semigroup $X^*$ and the group $G$ (viewed as a semigroup) together with the maps $(g, \alpha)\mapsto g\cdot \alpha$ and $(g, \alpha)\mapsto \left.g\right|_\alpha$ satisfy the conditions (ZS1)--(ZS8), and so we may form the Zappa-Sz\'ep product $X^*\bowtie G$. Furthermore, conditions (i)--(iv) from \S 5.3 are easily seen to hold, so $X^*\bowtie G$ is a right LCM semigroup. The semilattice of principal right ideals of $X^*\bowtie G$ is isomorphic to that of $X^*$ via the map $(\alpha, g)X^*\bowtie G \mapsto \alpha X^*$, and so one may identify $\J(X^*\bowtie G)$ with $\J(X^*)$, with inclusion order given by
\[
\alpha X^*\subset \beta X^* \Leftrightarrow \beta\text{ is a prefix of }\alpha.
\]
as before, principal right ideals are either disjoint or are comparable by inclusion. Hence as before, the unit space of the tight groupoid is homeomorphic to $\Sigma_X$, which is homeomorphic to the Cantor set. By \cite[Theorem 6.7]{BRRW14}, $\Q(X^*\bowtie G)\cong \mathcal{O}_{G, X}$. 

In general, $X^*\bowtie G$ is not cancellative, although it is embeddable into a group if and only if it is cancellative \cite[Theorem 5.5]{LW13}. We recall the following concepts from \cite{EPSep14}. Let $\alpha \in X^*$, and $g\in G$. Then $\alpha$ is said to be {\em strongly fixed by} $g$ if $g\cdot \alpha = \alpha$ and $\left.g\right|_\alpha = 1_G$, and we let 
\[
SF_g = \{\alpha\in X^*\mid \alpha\text{ strongly fixed by g}\}.
\]
Of course, if $\alpha\in SF_g$, then $\alpha\gamma\in SF_g$ for every word $\gamma\in X^*$. We will say that a strongly fixed word $\alpha$ is {\em minimal} by $g$ if $\alpha\in SF_g$ and no prefix of $\alpha$ is strongly fixed by $g$, and will denote this set by 
\[
MSF_g = \{\alpha\in X^*\mid \alpha\text{ minimal strongly fixed}\}\subset SF_g.
\]
The self-similar group $(G, X)$ is said to be {\em pseudo-free} if $SF_g$ is empty for all $g\neq 1_G$. A short calculation shows that $X^*\bowtie G$ is cancellative if and only if $(G,X)$ is pseudo-free (see \cite[Proposition 3.11]{LW13} or \cite[Lemma 3.2]{ES14}).

As mentioned earlier, our condition (H) is slightly weaker than right cancellativity, so one hopes that we can give conditions on $(G,X)$ which are equivalent to (H). If $(\alpha, g), (\beta, h)\in X^*\bowtie G$ meet, then there exists $(\gamma, k)\in X^*\bowtie G$ such that
\begin{eqnarray*}
(\alpha, g)(\gamma, k) &=& (\beta, h)(\gamma, k)\\
(\alpha(g\cdot \gamma), \left.g\right|_\gamma k) &=& (\beta(h\cdot \gamma), \left.h\right|_\gamma k),
\end{eqnarray*}
and since the action of $G$ on $X^*$ fixes lengths, we must have that $\alpha = \beta$,  $g\cdot \gamma = h\cdot\gamma$ and $\left.g\right|_\gamma=\left.h\right|_\gamma$. After noticing that the definition of a self-similar group implies that $\left.k\right|^{-1}_v = \left.k^{-1}\right|_{k\cdot v}$ for all $k\in G, v\in X^*$, this is easily seen to imply that $\gamma$ is strongly fixed by $g^{-1}h$. Hence, $(X^*\bowtie G)_{(\alpha, g), (\alpha, h)}$ is only nonempty when $g^{-1}h$ has a strongly fixed word, and
\[
(X^*\bowtie G)_{(\alpha, g), (\alpha, h)} = \{ (\gamma, k)\mid k\in G, \gamma\in SF_{g^{-1}h}\} = (X^*\bowtie G)_{(\emptyset, g), (\emptyset, h)}.
\]
Thus, $X^*\bowtie G$ will satisfy condition (H) if we can show that for all $g\in G\setminus \{1_G\}$, there exists a finite set $F\subset SF_g$ such that for all $\alpha\in SF_g$ there exists $\beta \in F$ such that $\beta X^*\cap \alpha X^*\neq \EM$. The following result appears in \cite{EPSep14} gives conditions for when this occurs.
\begin{lem}\cite[Theorem 12.2]{EPSep14}\label{MSFH}
Let $(G, X)$ be a self-similar group. Then $X^*\bowtie G$ satisfies condition (H) if and only if, for all $g\in G\setminus\{1_G\}$, the set $MSF_g$ is finite. 
\end{lem}
\begin{proof}
One easily sees that if $MSF_g$ is finite, it will satisfy the above condition, as each strongly fixed word must have a prefix which is minimal. Conversely, if such a finite $F$ exists, and $MSF_g$ is infinite, find a $\gamma\in MSF_g$ such that $|\gamma| > \max_{\alpha\in F}|\alpha|$. Then there must exist $\alpha\in F$ such that $\alpha X^*\cap \gamma X^*\neq \EM$, and since $|\gamma|> |\alpha|$, $\alpha$ must be a prefix of $\gamma$. But $\alpha$ is strongly fixed, and $\gamma$ is supposed to be minimal, so we have a contradiction. Hence $MSF_g$ is finite.
\end{proof}

We now address condition (EP). In this example, the core of $X^*\bowtie G$ coincides with the group of units of $X^*\bowtie G$, which is
\[
(X^*\bowtie G)_0 = U(X^*\bowtie G) = \{ (\emptyset, g)\mid g\in G\}
\]
and can be identified with the group $G$. The inverse semigroup \eqref{Sdef} has been previously constructed in \cite{EPSep14}, and generalizing their results there was an inspiration for this work. Let
\[
\s_{G, X} = \{ (\alpha, g, \beta)\mid \alpha,\beta\in X^*, g\in G\}.
\]
This set becomes an inverse semigroup when given the operation
\[
(\alpha, g, \beta)(\gamma, h, \nu) = \begin{cases}(\alpha (g\cdot\gamma'), \left.g\right|_{\gamma'}h, \nu), &\text{if }\gamma = \beta\gamma',\\ (\alpha, g(\left.h^{-1}\right|_{\beta'})^{-1}, \nu (h^{-1}\cdot\beta')), & \text{if } \beta = \gamma\beta',\\ 0 &\text{otherwise}\end{cases}
\]
with
\[
(\alpha, g, \beta)^* = (\beta, g^{-1}, \alpha).
\]
Then the map from our $\s$ to $\s_{G, X}$ given by 
\[
[(\alpha, g), (\beta, h)]\mapsto (\alpha, gh^{-1}, \beta)
\]
is an isomorphism of inverse semigroups, so from now on we will use this identification to discuss condition (EP). We note at this point that it follows from \cite[Corollary 6.4]{EPSep14}, $\Ct(\s_{G, X})\cong \mathcal{O}_{G, X}$, and so in this case our Theorem \ref{maintheorem} is already known. 

An element of $\s_{G, X}$ is an idempotent if and only if it is of the form $(\alpha, 1_G, \alpha)$. Identifying the core with $G$, we see that an idempotent $(\alpha, 1_G, \alpha)$ is weakly fixed by $g\in G$ if and only if, for all $\gamma\in X^*$, $(g\cdot \alpha)\gamma X^*\cap \alpha\gamma X^* \neq \EM$. By length considerations, this is equivalent to saying that $g\cdot \alpha = \alpha$ and for all $\gamma\in X^*$, $\left.g\right|_\alpha \cdot \gamma = \gamma$. If the action of $G$ on $X^*$ is {\em faithful} (which is to say that for all $g\in G\setminus \{1_G\}$ there exists $\alpha\in X^*$ such that $g\cdot \alpha \neq \alpha$), then this is equivalent to saying that $\alpha$ is strongly fixed by $g$. So, in the presence of faithfulness,
\[
(\alpha, 1_G, \alpha)\text{ weakly fixed by }(\emptyset, g, \emptyset) \Leftrightarrow \alpha\text{ strongly fixed by }g \Leftrightarrow (\alpha, 1_G,\alpha)\text{ fixed by }(\emptyset, g, \emptyset).
\]
Hence we have the following.
\begin{lem}\label{ssgEP}
Let $(G, X)$ be a faithful self-similar group, and let $g\in G$. Then $(\emptyset, g,\emptyset)$ satisfies condition (EP). 
\end{lem}

\indent We now come to the following result on self-similar groups. We note that it is not original to this work, and also follows from \cite[Proposition 5.5]{Nek09} and alternatively \cite[Proposition 17.1]{EPSep14}.

\begin{theo}\label{ssgclassification}
Let $(G, X)$ be a faithful self-similar group, suppose $G$ is amenable, and suppose that for all $g\in G\setminus\{1_G\}$, $MSF_g$ is finite. Then $\mathcal{O}_{G, X} \cong \Q(X^*\bowtie G)$ is nuclear, simple, and purely infinite. 
\end{theo}
\begin{proof}
Let $\s$ be as in \eqref{Sdef} for the semigroup $X^*\bowtie G$. Because $MSF_g$ is finite for all $g\in G\setminus \{1_G\}$, $X^*\bowtie G$ satisfies condition (H) by Lemma \ref{MSFH}. Because $G$ is amenable, we may apply \cite[Corollary 10.16]{EPSep14} to get that $\Q(X^*\bowtie G)$ is nuclear. Since nuclearity passes to quotients, this implies that $C^*_r(\Gt(S))$ is nuclear. Thus, by \cite[Theorem 5.6.18]{BO08}, $\Gt(\s)$ is amenable and so $C^*_r(\Gt(S))\cong C^*(\Gt(\s)) \cong \Q(X^*\bowtie G)$. By Lemma \ref{ssgEP}, every element of $\s_0$ satisfies (EP). Hence we may use Theorem \ref{Qsimple} to conclude that $\Q(X^*\bowtie G)$ is simple. Since we are assuming that $|X|>1$, $X^*\bowtie G\neq (X^*\bowtie G)_0$ implying that $\Gt(\s)$ is not the trivial groupoid, and so by Theorem \ref{purelyinfinite} we have that $\Q(X^*\bowtie G)$ is purely infinite.
\end{proof}

\begin{ex}{\bf The Odometer and Modified Odometer}

We will give two examples of faithful self-similar groups. For the first, let $X = \{0, 1\}$, let $\ZZ = \left\langle z\right\rangle$ be the group of integers with identity $e$ written multiplicatively. The {\em 2-odometer} is the self-similar group $(\ZZ, X)$ determined by
\[
z\cdot 0 = 1\hspace{1cm} \left.z\right|_0 = e
\]
\[
z\cdot 1 = 0\hspace{1cm} \left.z\right|_1 = z.
\]
If one views a word $\alpha\in X^*$ as a binary number (written backwards), then $z\cdot \alpha$ is the same as 1 added to the binary number for $\alpha$, truncated to the length of $\alpha$ if needed. If such truncation is not needed, $\left.z\right|_\alpha = e$, but if truncation is needed, $\left.z\right|_\alpha = z$. This self-similar group is faithful and pseudo-free \cite[Example 3.4]{ES14}. Hence $(\ZZ, X)$ satisfies the hypotheses of Theorem \ref{ssgclassification}, and so $\Q(X^*\bowtie \ZZ)$ is nuclear, simple, and purely infinite. In fact, this C*-algebra was shown in \cite[Example 6.5]{BRRW14} to be isomorphic to the C*-algebra $\Q_2$ defined in \cite{LL12}, and there the authors prove directly that it is nuclear, simple, and purely infinite. In \cite[Example 4.5]{ES14} we showed that this C*-algebra is isomorphic to a partial crossed product of the continuous functions on the Cantor set by the Baumslag-Solitar group $BS(1,2)$.

Since pseudo-freeness is stronger than what is needed to imply condition (H), we give a modified version of this example which is not pseudo-free but whose Zappa-Sz\'ep product does satisfy condition (H). To this end, let $X_B = \{0,1,B\}$, and let $\ZZ$ be written multiplicatively as before. Define
\[
z\cdot 0 = 1\hspace{1cm} \left.z\right|_0 = e
\]
\[
z\cdot 1 = 0\hspace{1cm} \left.z\right|_1 = z.
\]
\[
z\cdot B = B\hspace{1cm} \left.z\right|_B = e.
\]
One notices that the first two lines above are the same as in the previous example, but we have added a new symbol $B$ which is fixed by every group element; in fact, the word $B$ is strongly fixed by each group element. If one is wondering why we are calling this symbol ``$B$'', one could think of it as a {\bf B}rick wall past which no group element can travel or, if one pictures the odometer as acting like a car odometer, one could think of it as a {\bf B}roken digit. In any case, our new self-similar group $(X_B,\ZZ)$ is not pseudo-free.

If $\alpha\in MSF_{z^m}$ with $m>0$, then one quickly sees that $\alpha = \beta B$ for some $\beta\in \{0,1\}^*$, such that $z^m\cdot \beta = \beta$. Due to the description of the action as adding in binary, one sees that for a given $\beta\in \{0,1\}^*$, this equation is satisfied if and only if $k2^{|\beta|} = m$ for some $k>0$. Hence for a fixed $m$, only words $\beta$ of length less than $\log_2(m)$ could possibly satisfy $z^m\cdot \beta = \beta$. There are only finitely many such words, so for all $m>0$ the set $MSF_{z^m}$ is finite. The case $m<0$ is similar. Hence $X^*_B\bowtie \ZZ$ satisfies condtion (H). Direct application of Theorem \ref{ssgclassification} gives that $\Q(X^*_B\bowtie \ZZ)$ is nuclear, simple, and purely infinite.
\end{ex}

\appendix
\section{Appendix: The core of an inverse semigroup and topological freeness}\label{coreappendix}

In this brief appendix we establish a slightly more general form for Proposition \ref{EPcore}, beginning with a more general definition of the core of an inverse semigroup.

\begin{prop}\label{coreISG}
Let $S$ be an inverse semigroup with zero. Then the set
\[
S_0 := \{ s\in S\mid s^*se\neq 0\text{ and } ss^*e \neq 0\text{ for all }e\in E(S)\setminus\{0\}\}
\]
is closed under multiplication and the inverse operation, and so is an inverse subsemigroup of $S$.
\end{prop}
\begin{proof}
Suppose that $s, t\in S_0$, $e\in E(S)$, and $(st)^*ste = 0$. Then
\begin{eqnarray*}
t^*s^*ste &=& 0\\
\Rightarrow t^*s^*stet^* &=&0\\
\Rightarrow t^*tet^*s^*s &=& 0\\
\Rightarrow t^*tet^* &=& 0\\
\Rightarrow et^*t &=&0\\
\Rightarrow e &=& 0
\end{eqnarray*}
Similarly, if instead $st(st)^*e =0$ then one can show that this implies that $e=0$. Hence $st\in S_0$. Also, if $s\in S_0$ then $s^*$ clearly also is.
\end{proof}
\begin{defn}
Let $S$ be an inverse semigroup with zero. Then we call the inverse subsemigroup $S_0$ from Proposition \ref{coreISG} the {\em core} of $S$. 
\end{defn}

We are now able to prove a result similar to Proposition \ref{EPcore}, but only after making an additional hypothesis. In what follows, we consider the standard action $(\{D_e\}_{e\in E(S)}, \{\theta_s\}_{s\in S})$ of $S$ on its tight spectrum $\Et(S)$, $F_s$ denotes the fixed points for the element $s\in S$, and $TF_s$ denotes the trivially fixed points for $s$.

\begin{prop}
Let $S$ be an inverse semigroup with zero, and let $S_0$ be its core. Suppose further that for all $e\in E(S)\setminus\{0\}$, there exists $s\in S$ such that $ss^* = e$ and $s^*s = 1_S$, and that $\Gt(S)$ is Hausdorff. Then $\Gt(S)$ is essentially principal if and only if for all $s\in S_0$, every interior fixed point of $s$ is trivially fixed.
\end{prop}
\begin{proof}
The ``only if'' direction follows from \cite[Theorem 4.7]{EP14}. To prove the other direction, suppose that for all $s\in S_0$, $\mathring{F}_s\subset TF_s$, and suppose that $\Gt(S)$ is not essentially principal. Hence there must exist $s\in S$ such that $\mathring{F}_s\setminus TF_s$ is nonempty. Because we are assuming that $\Gt(S)$ is Hausdorff, $TF_s$ must be closed in $D_{s^*s}$, and so $\mathring{F}_s\setminus TF_s$ is open. Find an open set $U\subset \mathring{F}_s\setminus TF_s$, and find an ultrafilter $\xi\in U$. Then $\theta_s(\xi) = \xi$. We claim that we will be done if we can find $b\in S$ such that $bb^*\in \xi$, and $b^*sb\in S_0$. If such a $b$ exists, then the point $\theta_{b^*}(\xi)$ would be fixed by $b^*sb$, and so by assumption this point must be trivially fixed. Hence there is an idempotent $e\leqslant b^*sb$ with $\theta_{b^*}(\xi)\in D_e$. We would then have
\begin{eqnarray*}
b^*sbe &=&e\\
\Rightarrow bb^*sbeb^* &=& beb^*\\
\Rightarrow 0 \neq beb^* &\leqslant & bb^*s \leqslant s.
\end{eqnarray*}
Furthermore, since $\theta_{b^*}(\xi)\in D_e$, we must have that $\xi=\theta_b(\theta_{b^*}(\xi))\in D_{beb^*}$, and so $\xi$ is trivially fixed, which would be a contradiction. Hence, finding such a $b$ would prove our result.

So, we suppose that no such $b$ exists, fix $e\in \xi$. By assumption, we can find $b\in S$ such that $bb^* = e$ and $b^*b = 1_S$, and we are supposing that $b^*sb \notin S_0$. Hence there must exist a nonzero idempotent $f\in E(S)$ such that one of $b^*sb(b^*sb)^*f = 0$ or $(b^*sb)^*b^*sbf = 0$ holds. Suppose for a moment that the first holds. Then 
\begin{equation}\label{orthogonalcontradiction}
b^*sbb^*s^*bf = 0 \Rightarrow  bb^*sbb^*s^*bfb^* = 0 \Rightarrow (sbb^*s^*)(bfb^*) = 0. 
\end{equation}
Since $b^*b = 1_S$, $bfb^*$ cannot be 0, so find an ultrafilter $\eta_{bb^*}$ containing $bfb^*$. Since $bfb^*\leqslant bb^*$, we must have $bb^*\in \eta$. However by \eqref{orthogonalcontradiction}, $sbb^*s^*$ cannot be in $\eta$. So even if $\theta_{s}(\eta)$ is defined, it cannot be equal to $\eta$, because then it would have to contain $bfb^*$ and $sbb^*s^*$. The case where the second equation holds instead is similar. Hence for each $e = bb^*\in \xi$, we can construct an ultrafilter $\eta_e$ which contains $e$ but which is not fixed by $\theta_s$. Furthermore, if $e\in \xi$ then for all $k\leqslant e$, upwards closure of $\eta_k$ implies that $e\in \eta_k$, that is to say $\eta_k\in D_e$. Since the $D_e$ form a neighborhood base for ultrafilters \cite[Proposition 2.5]{EP14}, the $\{\eta_e\}_{e\in \xi}$ are a net converging to $\xi$, no elements of which are fixed by $\theta_s$. Since $\xi$ was assumed to be an interior fixed point, this is a contradiction. Hence we can find such a $b$, and we are done.   
\end{proof}

\bibliography{C:/Users/Charles/Dropbox/Research/bibtex}{}
\bibliographystyle{alpha}

{\small 
\textsc{Departamento de Matem\'atica, Campus Universit\'ario Trindade
CEP 88.040-900 Florian\'opolis SC, Brasil.}

Charles Starling: \texttt{slearch@gmail.com}}
\end{document}